\documentclass{amsart}
\usepackage[utf8]{inputenc}
\usepackage{a4wide}
\usepackage{amsmath,amsfonts,amssymb}
\usepackage{amsthm}
\usepackage{ifthen}
\usepackage{longtable}
\usepackage{array}
\usepackage{url}
\usepackage{enumerate}
\usepackage{enumitem}
\usepackage{hyperref}

\usepackage{titlesec}
\titleformat{\section}
  {\center\normalsize\bfseries}{\thesection.}{1em}{}

\newcommand{\Z}{\mathbb{Z}}
\newcommand{\Q}{\mathbb{Q}}

\newcommand{\eps}{\varepsilon}

\newcommand{\OO}{\mathcal{O}}

\DeclareMathOperator{\id}{id}
\DeclareMathOperator{\Gal}{Gal}

\newcommand{\caselI}[1]{\medskip\noindent\textit{#1}}
\newcommand{\caselII}[1]{\smallskip\textit{#1}}

\newtheorem*{theorem*}{Theorem}
\newtheorem{thm}{Theorem}
\newtheorem{lem}{Lemma}

\newtheorem{prop}{Proposition}
\newtheorem{con}{Conjecture}

\newtheorem*{claim*}{Claim}

	\newcounter{countknownthm}
	
\newtheorem{knownthm}[countknownthm]{Theorem}

\theoremstyle{definition}

\makeatletter
\@namedef{subjclassname@2020}{%
  \textup{2020} Mathematics Subject Classification}
\makeatother

\begin{document}
\title{On a family of unit equations over simplest cubic fields}
\subjclass[2020]{11D61, 11D75, 11J86} 
\keywords{Diophantine equations, unit equations, simplest cubic fields}
\thanks{The authors were supported by the Austrian Science Fund (FWF) under the project I4406.
}

\author[I. Vukusic]{Ingrid Vukusic}
\address{I. Vukusic,
University of Salzburg,
Hellbrunnerstrasse 34/I,
A-5020 Salzburg, Austria}
\email{ingrid.vukusic\char'100sbg.ac.at}

\author[V. Ziegler]{Volker Ziegler}
\address{V. Ziegler,
University of Salzburg,
Hellbrunnerstrasse 34/I,
A-5020 Salzburg, Austria}
\email{volker.ziegler\char'100sbg.ac.at}

\begin{abstract}
Let $a\in \Z$ and $\rho$ be a root of $f_a(x)=x^3-ax^2-(a+3)x-1$, then the number field $K_a=\Q(\rho)$ is called a simplest cubic field. In this paper we consider the family of unit equations $u_1+u_2=n$ where $u_1,u_2\in \Z[\rho]^*$ and $n\in \Z$. We completely solve the unit equations under the restriction $|n|\leq \max\{1,|a|^{1/3}\}$.
\end{abstract}

\maketitle

\section{Introduction}

We consider the family of parameterized polynomials
$$
f_{a}(x)=x^3-ax^2-(a+3)x-1,
$$
with parameter $a\in \Z$.
It is easy to show that for any $a\in \Z$ the polynomial $f_a$ is irreducible with a real root~$\rho$ and discriminant $D=(a^2+3a+9)^2$.
Therefore, the field $K=K_a=\Q(\rho)$ is Galois with cyclic Galois group generated by $\sigma$. In particular, the Galois group is given by $\Gal(K_a/\Q)=\langle \sigma\rangle$ where $\sigma (\rho)=-1-1/\rho$.
This family of cubic fields $K_a$ was first studied systematically by Shanks \cite{Shanks1974}, who called it a family of simplest cubic fields.
As already Shanks noted, $\rho$ and $\sigma(\rho)$ are multiplicatively independent units, which makes estimating the regulator of $K_a$ particularly easy.

This rather simple structure of the family of fields $K_a$, and in particular the simple form of a multiplicatively independent system of units, yield a nice underlying background structure to study various types of families of Diophantine equations. In particular, Thomas \cite{Thomas:1990} and subsequently Mignotte \cite{Mignotte:1993}  studied the family of Thue equations
$$ N_{K_a/\Q}(X-\rho Y)=X^3-aX^2Y-(a+3)XY^2-Y^3=\pm 1.$$
Families of twisted Thue equations were first studied by Levesque and Waldschmidt \cite{Levesque:2015}, who considered the family
$$N_{K_a/\Q}(X-\rho^t Y)=\pm 1.$$
We also want to mention that the integral points of the family of elliptic curves
$$\mathcal E_a:\quad Y^2=f_a(X)$$
have been studied by Duquesne \cite{Duquesne:2001}.

In this paper we want to study the family of unit equations
\begin{equation}\label{eq:unit}
u_1+u_2=n,\qquad n\in \Z,\;\; u_1,u_2\in \OO_{K_a}^*,
\end{equation}
where $\OO_{K_a}$ is the maximal order of $K_a$.

\section{Equivalent solutions, trivial solutions and main result}

Before we state our main result we make some observations. 
As mentioned above, for any $a\in \Z$ the number field $K_a=\Q(\rho)$ corresponding to $f_a(x)$ is Galois and the Galois group is cyclic and generated by $\sigma(\rho)=1-1/\rho$. 

Let $(u_1,u_2,n)$ be a solution to \eqref{eq:unit}. Then we can apply $\sigma$ to \eqref{eq:unit} and since $\sigma$ maps units to units and integers to integers, we see that $(\sigma(u_1),\sigma(u_2),n)$ is another solution to \eqref{eq:unit}. In fact,
\begin{align*}
 &(u_1,u_2,n),& &(\rho(u_1),\rho(u_2),n),& &(\rho^2(u_1),\rho^2(u_2),n),\\
 &(-u_1,-u_2,-n),& &(-\rho(u_1),-\rho(u_2),-n),& &(-\rho^2(u_1),-\rho^2(u_2),-n),\\
 &(u_2,u_1,n),& &(\rho(u_1),\rho(u_2),n),& &(\rho^2(u_1),\rho^2(u_2),n),\\
 &(-u_2,-u_1,-n),& &(-\rho(u_2),-\rho(u_1),-n),& &(-\rho^2(u_2),-\rho^2(u_1),-n)
\end{align*}
are each solutions to \eqref{eq:unit}. Note that these solutions are obtained by a combination of change of sign, permutation of $u_1$ and $u_2$ and application of some $\tau\in\Gal(K_a/\Q)$ to \eqref{eq:unit}. 
That is, with one solution usually come eleven more solutions that can easily be obtained from the first one. 
We shall call solutions that are obtained from $(u_1,u_2,n)$ in this way \textit{equivalent} to $(u_1,u_2,n)$. 
 
Next, note that there are some obvious solutions to \eqref{eq:unit}. Clearly, 
\begin{equation}\label{eq:trivial1}
	(1,1,2)
	\quad \text{and} \quad
	(u, - u , 0) \text{ with } u\in \OO_{K_a}^*
\end{equation}
are solutions to \eqref{eq:unit}. Moreover, one can check that $\rho$ and $\rho+1$ are also in $\OO_{K_a}^*$. Thus
\begin{equation}\label{eq:trivial2}
(\rho+1, -\rho, 1)
\end{equation}
is also a solution to \eqref{eq:unit}. We shall call all solutions that are equivalent to solutions from \eqref{eq:trivial1} and \eqref{eq:trivial2} \textit{trivial solutions}. All non-trivial solutions will be called \textit{sporadic}.

Further, note that $f_a(x)=-x^3f_{-a-3}(1/x)$, so if $\rho$ is a root of $f_a(x)=x^3-ax^2-(a+3)x-1$, then $1/\rho$ is a root of $f_{-a-3}(x)$. Hence $K_a=K_{-a-3}$. Therefore, we may assume that $a\geq -1$ when considering Equation \eqref{eq:unit}.

Finally, let us point out that in this paper we will not consider the units of the maximal order $\OO_{K_a}$ but of the order $\Z[\rho]$. The reason for this is that for the maximal order $\OO_{K_a}$ a parametrisation of the system of fundamental units is in general not known, while for $\Z[\rho]$ Thomas \cite{Thomas1979} showed that $\rho$ and $\sigma(\rho)$ form a system of fundamental units.
 However, we have $\Z[\rho]=\OO_{K_a}$ whenever the discriminant $D=(a^2+9a+3)^2$ is a square of a prime, and this is conjecturally (e.g. by Bunyakovsky's conjecture) infinitely often the case.

Now we state our main result.

\begin{thm}\label{thm:main}
Let $a\geq -1$, let $\rho$ be a root of $f_a(x)$ and consider the Diophantine Equation
\begin{equation}\label{eq:main}
 u_1+u_2=n,
 \qquad 
 u_1,u_2\in \Z[\rho]^*, \;
 n\in \Z, \; |n|\leq \max\{|a|^{1/3},1\}.
\end{equation}
Then for $a\geq 3$ Equation \eqref{eq:main} has only trivial solutions. For $a\leq 2$ every solution is either trivial or equivalent to one of the 10 sporadic solutions $(u_1,u_2,1)$ presented in Table \ref{tab:sporadicSols} in Section \ref{sec:smallSols}.
\end{thm}

Let us outline the strategy of our proof. As mentioned above, $\epsilon:=\rho$ and $\delta:=-\sigma(\rho)$ form a fundamental system of units of $\Z[\rho]$. Therefore, we can write
$$u_1=\pm \epsilon^{x_1}\delta^{y_1},\qquad u_2=\pm \epsilon^{x_2}\delta^{y_2}$$
for some $x_1,x_2,y_1,y_2\in \Z$. Assume that $(u_1,u_2,n)$ is a solution to \eqref{eq:main} and let $X:=\max\{|x_1|,|x_2|,|y_1|,|y_2|\}$. Our first goal is to show that it is always possible to find a solution $(\bar u_1,\bar u_2,n)$ equivalent to $(u_1,u_2,n)$ such that $|\bar u_2|>a^{X/2}$ (Section \ref{sec:embedding}). In Section \ref{sec:specialCases} we deal with some special cases that yield the three families of trivial solutions.
Established these preliminary results we apply lower bounds for linear forms in logarithms to obtain an upper bound for $X$ of the form $X\ll \log a \, (\log \log a)^2$ (Section \ref{Sec:upper}). On the other hand, using a similar trick as in \cite{Thomas:1990} we obtain a lower bound for $X$ of the form $X\gg a \log a$ (see Section \ref{Sec:lower}). These two bounds yield an absolute upper bound for $a$. Using continued fractions we show that there are no sporadic solutions provided that $a> 100$. The case $a\leq 100$ is dealt with in the next section.

\section{Small solutions}\label{sec:smallSols}

For technical reasons we want to exclude small values of the parameter $a$ and start by proving the following proposition:

\begin{prop}\label{prop:smallSols}
For $a\leq 100$ the only solutions to Equation \eqref{eq:main} are those stated in Theorem~\ref{thm:main}.
\end{prop}

\begin{proof}
The computations were done with Magma \cite{Magma}. The function \verb|UnitEquation| was used to solve the unit equation \eqref{eq:main} over the maximal order of $K_a$ for $-1 \leq a \leq 100$ and $1 \leq n\leq \max\{|a|^{1/3},1\}$. 
This took only a couple of seconds and revealed a total of $693$ solutions. 

Then the trivial solutions were filtered out as well as those not lying in $\Z[\rho]$. Thus 60 solutions remained and in each of them we had $n=1$. Up to equivalence they correspond to the $10$ sporadic solutions presented in Table~\ref{tab:sporadicSols}.
\end{proof}

In Table \ref{tab:sporadicSols} the sporadic solutions to Equation \eqref{eq:main} are represented in terms of $\rho$ as well as powers of fundamental units $\eps=\rho$ and $\delta=-\sigma(\rho)$.

\begin{table}[h]
\[
\begin{tabular}{ >{$}c<{$}  >{$}r<{$} @{\hspace{2pt}} >{$}l<{$} >{$}r<{$} @{\hspace{2pt}}
	>{$}l<{$}}
\hline
a & &u_1 & &u_2  \\	
\hline
-1 &
-\rho -1  & =  - \eps^{ 1 } \delta^{ 1 } &
\rho +2  & =   \eps^{ 0 } \delta^{ 2 } \\
&
-\rho^2+3  & =   \eps^{ -1 } \delta^{ 1 } &
\rho^2 -2  & =  - \eps^{ -1 } \delta^{ -1 } \\
&
-\rho^2+\rho +1  & =   \eps^{ 1 } \delta^{ -1 } &
\rho^2 -\rho   & =   \eps^{ 0 } \delta^{ -2 } \\
&
-549\rho^2 -305\rho +1234  & =  - \eps^{ -11 } \delta^{ -8 } &
549\rho^2 +305\rho -1233  & =   \eps^{ -8 } \delta^{ 3 } \\
&
-6\rho^2 -3\rho +14  & =   \eps^{ -3 } \delta^{ 1 } &
6\rho^2 +3\rho -13  & =   \eps^{ -4 } \delta^{ -3 } \\
&
-3\rho^2 -2\rho +7  & =  - \eps^{ -3 } \delta^{ -2 } &
3\rho^2 +2\rho -6  & =   \eps^{ -2 } \delta^{ 1 } \\
\hline
0 &
-21\rho^2 +7\rho +61  & =  - \eps^{ -5 } \delta^{ -2 } &
21\rho^2 -7\rho -60  & =   \eps^{ -2 } \delta^{ 3 } \\
&
-2\rho^2 +\rho +6  & =   \eps^{ -1 } \delta^{ 1 } &
2\rho^2 -\rho -5  & =   \eps^{ -2 } \delta^{ -1 } \\
\hline
1 &
-21\rho^2 +27\rho +77  & =   \eps^{ -1 } \delta^{ 3 } &
21\rho^2 -27\rho -76  & =   \eps^{ -4 } \delta^{ -1 } \\
\hline
2 &
-603\rho^2 +1340\rho +2718  & =   \eps^{ -1 } \delta^{ 5 } &
603\rho^2 -1340\rho -2717  & =   \eps^{ -6 } \delta^{ -1 } \\
\hline
\end{tabular}
\]
\caption{Sporadic solutions $(u_1,u_2,1)$}
\label{tab:sporadicSols}
\end{table}

From now on, we assume that $a> 100$.

\section{Further notations and finding a good solution in the equivalence class}\label{sec:embedding}

First, we note that $f_a(x)$ has three real roots $\rho=\rho_1,\rho_2,\rho_3$. Let us choose the ordering of the roots such that $\rho$ is the largest root,
$\rho_2=\sigma(\rho)=-1 - 1/\rho$ and $\rho_3=\sigma^2(\rho)=-1/(1+\rho)$.
Then we have the following estimates:
\begin{align*}
	a+1 < & \rho_1 < a+2,\\	
	-1 - \frac{1}{a+2} < & \rho_2 < -1 - \frac{1}{a+3},\\
	-\frac{1}{a+2} < & \rho_3 < -\frac{1}{a+3}.
\end{align*}
The estimates can easily be checked e.g. by inserting the upper and lower bounds in $f_a(x)$ and observing that the sign changes. 

Since $\rho_1\rho_2\rho_3=1$, the roots $\rho_1,\rho_2,\rho_3$ are indeed units in the maximal order $\OO_{K_a}$ of $K_a$.

Next, observe that $\sigma(\rho)=-\rho^2+a\rho + a+2 \in \Z[\rho]$, which implies $\Z[\rho]=\Z[\rho_2]=\Z[\rho_3]$. Thus $\rho_1,\rho_2,\rho_3$ are indeed units in the order $\Z[\rho]$. Moreover, Thomas \cite{Thomas1979} showed that
$\epsilon:=\rho_1$ and $\delta:=-\rho_2$ form a fundamental system of units of the order $\Z[\rho]$.

\medskip

Now let $(u_1,u_2,n)$ be a solution to Equation \eqref{eq:main}. In the following we will find a $\tau\in \Gal(K_a/\Q)$ such that $|\tau(u_2)|$ is ``large''.

 Since $\eps$ and $\delta$ are a fundamental system of units, we can write
\[
	u_1= \pm \eps^{x_1}\delta^{x_1}, \quad
	u_2= \pm \eps^{x_2}\delta^{y_2}.
\]
Consider the set of all solutions $(\bar u_1,\bar u_2, n)$ equivalent to $(u_1,u_2,n)$ and let $X$ be the maximum of all absolute values of exponents of $\epsilon$ and $\delta$ in $\bar u_1$ and $\bar u_2$. By choosing an appropriate solution equivalent to $(u_1,u_2,n)$ we may assume without loss of generality that
\[ 
	X=\max\{|x_2|, |y_2|\}.
\]
Next, we find a ``good'' $\tau\in \Gal(K_a/\Q)$. Therefore we prove first

\begin{lem}\label{lem:embedding}
There is a $\tau\in \Gal(K_a/\Q)$ such that 
$\tau(u_2)=\bar u_2=\pm \eps^{\bar{x}_2}\delta^{\bar{y}_2}$ 
and the exponents satisfy one of the following conditions:
\begin{enumerate}[label=(\alph*)]
\item \label{case:X-X}$\bar{x}_2=X$ and $\bar{y}_2\geq -X$,
\item \label{case:X2-X2}$\bar{x}_2\geq X/2$ and $\bar{y}_2\geq -X/2$.
\end{enumerate}
\end{lem}
\begin{proof}
We distinguish between several cases, according to the size of the exponents $x_2$ and~$y_2$. 

\caselI{Case 1:} $X=x_2$. Then condition \ref{case:X-X} is immediately satisfied with $\tau=\id_{K_a}$.

\caselI{Case 2:} $X=-x_2$. 
We distinguish between two subcases.

\caselII{Case 2.1:} $y_2 \geq -X/2$. 
If we choose $\tau=\sigma^2$, then we have
\[
	|\tau(u_2)|
	=|\rho_3|^{x_2}|\rho_1|^{y_2}
	=|\rho_1 \rho_2|^{-x_2}|\rho_1|^{y_2}
	=\eps^{-x_2+y_2} \delta^{-x_2}
	=\eps^{\bar{x}_2}\delta^{\bar{y}_2}.
\]
Thus condition \ref{case:X2-X2} is satisfied:
\begin{align*}
	\bar{x}_2
	&=-x_2 + y_2 
	\geq X - X/2 
	= X/2,\\
	\bar{y}_2
	&= -x_2 = X \geq -X/2.
\end{align*}

\caselII{Case 2.2:} $y_2 < - X/2$. In this case we choose $\tau=\sigma$. Then we have
\[
	|\tau(u_2)|
	= |\rho_2|^{x_2} |\rho_3|^{y_2}
	= |\rho_2|^{x_2} |\rho_1\rho_2|^{-y_2}
	= \eps^{-y_2} \delta^{-y_2+x_2}
	=\eps^{\bar{x}_2}\delta^{\bar{y}_2}
\]
and condition \ref{case:X2-X2} is satisfied:
\begin{align*}
	\bar{x}_2
	&=-y_2
	> X/2,\\
	\bar{y}_2
	&= -y_2 + x_2
	> X/2 - X
	= -X/2.
\end{align*}

\caselI{Case 3:} $X=y_2$. Again, we distinguish between two subcases.

\caselII{Case 3.1:}
$x_2 \geq X/2$. Then condition \ref{case:X2-X2} is immediately satisfied with $\tau=\id_{K_a}$.

\caselII{Case 3.2:} 
$x_2 < X/2$. 
If we take $\tau=\sigma^2$, then as in Case 2.1 we have 
\[
	|\tau(u_2)|
	=\eps^{-x_2+y_2} \delta^{-x_2}
	=\eps^{\bar{x}_2}\delta^{\bar{y}_2}
\]
and condition \ref{case:X2-X2} is satisfied:
\begin{align*}
	\bar{x}_2
	&=-x_2 + y_2 
	> -X/2 + X
	= X/2,\\
	\bar{y}_2
	&= -x_2
	> -X/2.
\end{align*}

\caselI{Case 4:} $X=-y_2$. We choose $\tau=\sigma$. Then as in Case~2.2 we have 
\[
	|\tau(u_2)|
	= \eps^{-y_2} \delta^{-y_2+x_2}
	=\eps^{\bar{x}_2}\delta^{\bar{y}_2}
\]
and condition \ref{case:X-X} is satisfied:
\begin{align*}
	\bar{x}_2
	&=-y_2
	=X,\\
	\bar{y}_2
	&= -y_2 + x_2
	\geq X - X = 0 \geq - X.
\end{align*}
\end{proof}

Now we can prove that $|\bar u_2|$ is ``large''.


\begin{lem}\label{lem:u2gross}
Let $\bar u_2 = \tau (u_2)$ be as stated in Lemma \ref{lem:embedding}. Then we have
\[
	|\bar u_2|>a^{X/2}.
\]
\end{lem}
\begin{proof}
First assume that condition \ref{case:X-X} of Lemma \ref{lem:embedding} is satisfied, i.e. $\bar x_2=X$ and $\bar y_2 \geq -X$. Then we have
\[
	|\bar u_2|
	=\eps^{\bar x_2}\delta^{\bar y_2}
	> (a+1)^{X}\left(1+\frac{1}{a+2}\right)^{-X}
	= \left( \frac{(a+1)(a+2)}{a+3}\right)^{X}
	> a^{X}.
\]
Now assume that condition \ref{case:X2-X2} is satisfied, i.e. $\bar x_2\geq X/2$ and $\bar y_2\geq -X/2$. Then we have
\[
	|\bar u_2| 
	= \eps^{\bar x_2}\delta^{\bar y_2}
	> (a+1)^{X/2} \left(1+\frac{1}{a+2}\right)^{-X/2}
	= \left( \frac{(a+1)(a+2)}{a+3}\right)^{X/2}
	> a^{X/2}.
\]

\end{proof}

Now consider the solution $(\bar u_1 , \bar u_2, n)= (\tau(u_1),\tau(u_2),n)=(\pm \eps^{\bar x_1} \delta^{\bar y_1}, \pm \eps^{\bar x_2} \delta^{\bar y_2},n)$.
Since $\max\{|\bar x_1|,|\bar x_2|,|\bar y_1|,|\bar y_2|\}\leq X$ we immediately obtain the following proposition. (Note that after a possible change of signs we may assume $\bar u_1$ to be positive.)

\begin{prop}\label{prop:good}
Every solution to \eqref{eq:main} is equivalent to a solution $(u_1,u_2,n)$ satisfying
$|u_2|>a^{X/2}$, where $X=\max\{|x_1|,|x_2|,|y_1|,|y_2|\}$ and
$u_1=\eps^{x_1}\delta^{y_1}$, $u_2=\pm \eps^{x_2}\delta^{y_2}$.
\end{prop}

From now on, we will only consider the solutions from Proposition \ref{prop:good}, i.e. we will always assume without loss of generality that a solution  $(u_1,u_2,n)$ is of the form $u_1=\eps^{x_1}\delta^{y_1}$, $u_2=\pm \eps^{x_2}\delta^{y_2}$ and that $|u_2|>a^{X/2}$. Moreover, $X$ shall denote the maximum of the exponents $X=\max\{|x_1|,|x_2|,|y_1|,|y_2|\}$.

\section{Some special cases}\label{sec:specialCases}

In this section we treat the special case $x_1=x_2$, finding the trivial solutions to Equation~\eqref{eq:main}.

\begin{lem}\label{lem:einfacheL}
Let $(u_1,u_2,n)=(\eps^{x_1}\delta^{y_1}, \pm \eps^{x_2}\delta^{y_2},n)$ be a solution to \eqref{eq:main} with $x_1=x_2$ and $y_1-y_2 \in \{0,\pm 1\}$. Then $(u_1,u_2,n)$ is a trivial solution.
\end{lem}
\begin{proof}
First, assume that $x_1=x_2$ and $y_1-y_2=0$. Then we have $u_1=\pm u_2$. 
If $u_1 = -u_2$, then we have the trivial solution $(u_1, -u_1 ,0)$.
If $u_1= u_2$, then we have $2 \eps^{x_1} \delta^{y_1} = n$, which is only possible if $x_1=y_1=0$, i.e. $u_1=u_2=1$ and $n=2$.

Now consider the case that $x_1=x_2$ and $y_1-y_2=1$. Then we have 
\[
	\eps^{x_1}\delta^{y_2+1} \pm \eps^{x_1}\delta^{y_2}
	=\eps^{x_1}\delta^{y_2}(\delta \pm 1)
	=n.
\]  
On the one hand, one can check that $N_{K_a/\Q}(\delta+1)=-2a-3$. Since $\eps$ and $\delta$ are units and $n\in \Z$ with $|n|\leq a^{1/3}$, we obtain a contradiction considering the norm of the above equation:
\[
	2a+3 
	=|N_{K_a/\Q}(u_1+u_2)|
	=|N_{K_a/\Q}(n)|
	=|n|^3
	\leq a.
\]
On the other hand, $N_{K_a/\Q}(\delta -1)=1$ and there is a solution for $n=1$:
We have	$\eps^{x_1}\delta^{y_2}(\delta-1)=1$ if and only if $\eps^{x_1}\delta^{y_2}=1/(\delta-1)=\eps $, i.e. $x_1=x_2=1$, $y_1=1$ and $y_2=0$. Thus we have $u_1= \eps \delta=\eps + 1$ and $u_2= -\eps$, which is the trivial solution $(\rho+1,-\rho,1)$.

The case $x_1=x_2$ and $y_1-y_2=-1$ is analogous to the above case.
\end{proof}

In the more general cases we will also argue via the norm, using the following lemma.

\begin{lem}\label{lem:norm}
For $x\in \Z\setminus\{0,\pm1 \}$ we have
\[
	|N_{K_a/\Q}(\eps^x \pm 1)|>a.
\]
\end{lem}
\begin{proof}
First, assume that $x\geq 2$. Then
\begin{align*}
	|N_{K_a/\Q}(\eps^x \pm 1)|
	&= |\rho_1^x \pm 1|\cdot|\rho_2^x \pm 1|\cdot|\rho_3^x \pm 1|\\
	&> ((a+1)^2 - 1) \left( \left( 1 + \frac{1}{a+3}\right)^2 - 1 \right)
		\left( 1- \left( \frac{1}{a+2}\right)^2 \right)\\
	&=\frac{a(1+a)(7+2a)}{(2+a)(3+a)}
	> a.
\end{align*}
For $x\leq -2$ we have
\begin{align*}
	|N_{K_a/\Q}(\eps^x \pm 1)|
	&> (1- (a+1)^{-2}) \left( 1 - \left( 1 + \frac{1}{a+3}\right)^{-2} \right)
		\left( \left( \frac{1}{a+2}\right)^{-2} - 1 \right)\\
	&=\frac{a(a+2)(a+3)(2a+7)}{(a+1)(a+4)^2}
	> a.
\end{align*}
\end{proof}

Now we can prove

\begin{prop}\label{prop:specialCases}
Any solution $(u_1,u_2,n)=(\eps^{x_1}\delta^{y_1}, \pm \eps^{x_2}\delta^{y_2},n)$ to Equation \eqref{eq:main} with $x_1=x_2$ is trivial.
\end{prop}
\begin{proof}
By Lemma \ref{lem:einfacheL} we only need to check the cases where $x_1=x_2$ and $y_1-y_2 \notin\{0,\pm 1\}$.
Under these assumptions Equation \eqref{eq:main} becomes
\[
	n
	= \eps^{x_1} \delta^{y_1} \pm \eps^{x_1} \delta^{y_2}
	= \eps^{x_1} \delta^{y_2} (\delta^{y_1-y_2} \pm 1).
\]
Considering the norm, we obtain a contradiction with Lemma \ref{lem:norm}:
\[
	a 
	\geq |n|^3 
	=| N_{K_a/\Q}(\delta^{y_1-y_2} \pm 1)|
	>a.
\]
\end{proof}

From now on, we assume that $x_1\neq x_2$ and in particular $X\geq 1$.

\section{Obtaining an upper bound for $X$}\label{Sec:upper}

In this section we use a lower bound for linear forms in logarithms by  Laurent \cite[Cor. 2]{Laurent2008} in order to obtain a bound of the form $X\ll \log a \,(\log \log a)^2$. To state Laurent's result we have to introduce the notion of height. Let $\alpha \neq 0$ be an algebraic number of degree $d$ and let
\[
	a_0(x-\alpha_1)\cdots (x-\alpha_d) \in \Z[x]
\]
be the minimal polynomial of $\alpha$. Then the absolute logarithmic Weil height is defined by
$$h(\alpha):=\frac 1d \left(\log |a_0|+\sum_{i=1}^d \max\{0,\log|\alpha_i|\}\right).$$ 
With this notation we have

\begin{knownthm}[Laurent]\label{thm:Laurent}
Let $\alpha_1, \alpha_2 \geq 1$ be two multiplicatively independent real algebraic numbers, let $b_1$ and $b_2$ be two positive integers and let 
\[
	\Lambda := b_2 \log \alpha_2 - b_1 \log \alpha_1.
\]
Then
\[
	\log |\Lambda|
	\geq - 17.9 D^4 (\max\{ \log b' + 0.38, 30/D, 1\})^2 \log A_1 \log A_2, 
\]
where
\begin{align*}
	D&=[\Q(\alpha_1,\alpha_2):\Q],\\
	b'&=\frac{b_1}{D \log A_2} + \frac{b_2}{D \log A_1},\\
	\log A_i &\geq \max \{ h(\alpha_i),|\log \alpha_i|/D,1/D\}
		\qquad (i=1,2).
\end{align*}
\end{knownthm}

\begin{prop}\label{prop:upperBound}
Let $a > 100$ and $(u_1,u_2,n)$ be a solution with the properties from Proposition~\ref{prop:good} and $x_1 \neq x_2$. 
Then for the maximum of exponents $X$ we have the upper bound
\begin{equation}\label{eq:upperBound}
	X < 343 \log a \cdot (10 +  1.7 \log \log a)^2.
\end{equation} 
\end{prop}
\begin{proof}
First, note that Inequality \eqref{eq:upperBound} is trivially fulfilled if $X\leq 200\,000$. Thus we may assume $X> 200\,000$.

We start with the unit equation $u_1+u_2=n$.
Taking absolute values and dividing by $|u_2| > a^{X/2}$
we obtain
\begin{equation*}\label{eq:uB_preLinform}
	\left|\frac{u_1}{u_2} - 1 \right| 
	= |\eps^{x_1-x_2} \delta^{y_1-y_2} \pm 1 | 
	= \frac{|n|}{|u_2|}
	< \frac{a^{1/3}}{a^{X/2}}. 
\end{equation*} 
Since we are assuming $X\geq 1$ and $a>100$,
the right-hand side of the above inequality is smaller than $0.5$. Since $\eps$ and $\delta$ are positive, the ``+'' sign cannot hold and we have
\begin{equation*}
	|\eps^{x_1-x_2} \delta^{y_1-y_2} - 1 |
	<  \frac{a^{1/3}}{a^{X/2}}.
\end{equation*}
Since $|\log x|< 2 |x-1|$ for $|x-1|<0.5$, this implies
\begin{equation}\label{eq:uB_linform}
	|\Lambda|	
	:=|(x_1-x_2)\log \eps - (y_1-y_2) \log \delta|
	< \frac{2a^{1/3}}{a^{X/2}}.
\end{equation}
Note that $\eps, \delta$ are multiplicatively independent and larger than 1. Further, note that $x_1-x_2$ and $y_1-y_2$ are integers.
Since $|\Lambda|<0.5$, $\log \eps>1$ and $x_1-x_2\neq 0$, it is easy to see that $y_1-y_2$ is nonzero and of the same sign as $x_1-x_2$.
Therefore, we may apply \hyperref[thm:Laurent]{Laurent's} theorem. First, we estimate the height of $\alpha_1:=\eps$ and $\alpha_2:=\delta$:
\begin{align*}
	h(\alpha_1)=h(\alpha_2)
	= h(\rho_1)
	&=\frac{1}{3}(\log 1 + \log \max\{1, |\rho_1|\} + \log \max\{1, |\rho_2|\} + \log \max\{1, |\rho_3|\} )\\
	&=\frac{1}{3}(\log|\rho_1| + \log |\rho_2|\}\\
	&< \frac{1}{3}\left(\log (a+2) +\log \left(1+ \frac{1}{a+2}\right)\right)
	= \frac{1}{3}\log (a+3).
\end{align*}
Now we have
\begin{align*}
D &= [\Q(\alpha_1,\alpha_2):\Q]=3,\\
\log A_1 =\log A_2 
	&:= \frac{1}{3}\log (a+3)
	\geq \max \{ h(\alpha_i),|\log \alpha_i|/D,1/D\},\\ 
b' &= \frac{|x_1-x_2|}{D\log A_2} + \frac{|y_1-y_2|}{D\log A_1} =\frac{|x_1-x_2|+|y_1-y_1|}{\log(a+3)}.
\end{align*}
Thus we obtain
\begin{multline}\label{eq:uB_LaurentApplied}
	-17.9 \cdot 3^4 \left( \max \left\{ 
		\log \left( \frac{|x_1-x_2|+|y_1-y_1|}{\log(a+3)} \right) + 0.38,
		\frac{30}{3}, 1
	\right\} \right)^2
	\left(\frac{1}{3}\log(a+3) \right)^2\\
	\leq \log |\Lambda|
	< \log 2 + \frac{1}{3}\log a - \frac{X}{2}\log a.
\end{multline}
Next, note that
\begin{align*}
	\log \left( \frac{|x_1-x_2|+|y_1-y_2|}{\log(a+3)} \right) + 0.38
	&\leq \log \left(\frac{4X}{\log(a+3)}\right) + 0.38\\
	&= \log 4 + \log X - \log \log (a+3) + 0.38\\
	&\leq \log 4 + \log X - \log \log 103 + 0.38\\
	< \log X + 0.24.  
\end{align*}
Moreover, since we are assuming $X>200\,000$, we have in particular $\log X > 12$, so
\[
	\max \left\{\log \left( \frac{|x_1-x_2|+|y_1-y_1|}{\log(a+3)} \right) + 0.38,
		\frac{30}{3}, 1 \right\}
	<\log X + 0.24
	< 1.02 \log X.
\]
Therefore, we obtain from \eqref{eq:uB_LaurentApplied}
\[
	-167.61 (\log X)^2 (\log (a+3))^2
	<  \log 2 + \frac{1}{3}\log a - \frac{X}{2}\log a,
\]
which implies
\[
	\frac{X}{2}\log a
	< 167.61 (\log X)^2 (\log (a+3))^2 + \log 2 + \frac{1}{3}\log a
	< 168 (\log X)^2 (\log (a+3))^2.
\]
Multiplying by $2(\log a)^{-1}$ we obtain
\begin{equation}\label{eq:uB_preBound}
\begin{split}
	X 
	&< 168 (\log X)^2 (\log (a+3))^2 \cdot 2 (\log a)^{-1} \\
	&< 168 (\log X)^2 (1.01 \log a)^2 \cdot 2 (\log a)^{-1} \\
	&< 343 \log a \cdot (\log X)^2 . 
\end{split}	
\end{equation}
This implies (note that we are assuming $X>200\,000$)
\begin{align*}
	\log X 
	&< \log 343 + \log \log a + 2\log \log X \\
	&< 5.9 +  \log \log a + 0.41 \log X
\end{align*}
and therefore
\[
	\log X 
	< \frac{5.9}{0.59} + \frac{\log \log a}{0.59}
	< 10 +  1.7 \log \log a. 
\]
Thus \eqref{eq:uB_preBound} implies
\[
	X < 343 \log a \cdot (10 +  1.7 \log \log a)^2.
\]
\end{proof}

\section{Obtaining a lower bound for $X$}\label{Sec:lower}

Next, we do some elementary estimations in order to obtain the following lower bound for $X$.

\begin{prop}\label{prop:lowerBound}
Let $a > 100$ and $(u_1,u_2,n)$ be a solution with the properties from Proposition~\ref{prop:good} and $x_1\neq x_2$. 
Then for the maximum of exponents $X$ we have the lower bound
\begin{equation*}
 X > \frac{1}{2}(a+2)(\log (a+1) - \log 2).
\end{equation*}
\end{prop}
\begin{proof}
As in the proof of Proposition \ref{prop:smallSols}, assuming $X\geq 1$ and $a> 100$ we have
\[
	|\eps^{x_1-x_2} \delta^{y_1-y_2}- 1 |
	< \frac{a^{1/3}}{a^{X/2}}
	< 0.5.
\]	
Since we assume $x_1-x_2\neq 0$, we can distinguish between the two cases $x_1-x_2\leq -1$ and $x_1-x_2\geq 1$.

\caselI{Case 1:} $x_1-x_2 \leq -1$. Then we have
\begin{align*}
	0.5
	< \eps^{x_1-x_2} \delta^{y_1-y_2}
	\leq (a+1)^{-1} \delta^{y_1-y_2},
\end{align*}
so we must have $y_1-y_2 > 0$ and
\[
	\frac{1}{2} 
	< (a+1)^{-1} \left( 1+ \frac{1}{a+2}\right)^{y_1-y_2}
	\leq (a+1)^{-1} \left( 1+ \frac{1}{a+2}\right)^{2X}.
\]
This implies
\begin{align*}
	X
	> \frac{\log (a+1) - \log 2}{2 \log \left( 1+ \frac{1}{a+2}\right)}
	> \frac{1}{2}(a+2)(\log (a+1) - \log 2),
\end{align*}
where we used that $\log (1+x)<x$ and therefore $\frac{1}{\log(1+x)}>\frac 1x$ for all $x>0$. 

\caselI{Case 2:} $x_1-x_2 \geq 1$. In this case we have
\[
	1.5
	> \eps^{x_1-x_2} \delta^{y_1-y_2}
	\geq (a+1)\, \delta^{y_1-y_2},
\]
so we must have $y_1-y_2<0$ and
\[
	1.5 
	> (a+1) \left( 1+ \frac{1}{a+2}\right)^{y_1-y_2}
	\geq (a+1) \left( 1+ \frac{1}{a+2}\right)^{-2X}.
\]
This implies
\begin{align*}
	X
	> \frac{\log (a+1) + \log 1.5}{2 \log \left( 1+ \frac{1}{a+2}\right)}
	> \frac{1}{2}(a+2)(\log (a+1) - \log 2).
\end{align*}
\end{proof}

\section{Finishing the proof of Theorem \ref{thm:main}}

Combining Proposition \ref{prop:upperBound} and Proposition \ref{prop:lowerBound} we obtain
an absolute bound for~$a$:

\begin{prop}\label{prop:largeBound}
There are no sporadic solutions for 
$a > 1.48 \cdot 10^5$.
\end{prop}
\begin{proof}
The statement follows from Propositions \ref{prop:good}, \ref{prop:specialCases}, \ref{prop:upperBound} and \ref{prop:lowerBound}. The bound is obtained by solving the inequality
\[
	\frac{1}{2}(a+2)(\log (a+1) - \log 2)
	< 343 \log a \cdot (10 +  1.7 \log \log a)^2.
\]
\end{proof}

Finally, we use continued fractions to deal with all $100<a\leq 1.48\cdot 10^5$ and thus finish the proof of Theorem \ref{thm:main}.

Recall Inequality \eqref{eq:uB_linform} from the proof of Proposition \ref{prop:smallSols}:
\[
	|(x_1-x_2)\log \eps - (y_1-y_2) \log \delta|
	< 2a^{\frac{1}{3}-\frac{X}{2}}.
\]
Note that by Proposition \ref{prop:specialCases} we may assume $x_1-x_2\neq 0$ and therefore it is easy to check that $y_1-y_2\neq 0$ as well.
Dividing by $\log \eps > \log a$ we obtain
\begin{equation}\label{eq:red_ineq}
	\left| (x_1-x_2)-(y_1-y_2)\frac{\log \delta}{\log \eps} \right|
	< \frac{2a^{\frac{1}{3}-\frac{X}{2}}}{\log a}.
\end{equation}
We proceed for each $a$ with $100< a \leq 1.48 \cdot 10^5$ as follows.
First we compute $\frac{\log \delta}{\log \eps}$ and find the smallest convergent $\frac{p}{q}$ to $\frac{\log \delta}{\log \eps}$ such that
\[
	q
	\geq 2 \cdot 343 \log a \cdot (10 +  1.7 \log \log a)^2
	> 2X
	\geq |y_1-y_2|.
\]
Then by the best approximation property of continued fractions we have
\[
	c:= \left|p - q\, \frac{\log \delta}{\log \eps}\right|
	<\left|(x_1-x_2) - (y_1-y_2) \frac{\log \delta}{\log \eps}\right|
	< \frac{2a^{\frac{1}{3}-\frac{X}{2}}}{\log a},
\]
which implies a new bound
\[
	X
	< \frac{2}{\log a} \log \left( \frac{2a^{1/3}}{c \log a} \right) .
\]
On the other hand, we have
\[
	X> \frac{1}{2}(a+2)(\log (a+1) - \log 2)
\]
and if upper and lower bounds for $X$ contradict each other, we deduce that no non-trivial solution for this specific $a$ exists. Indeed, the computations for $100 < a \leq 1.48 \cdot 10^5$ revealed that the bounds contradict each other for every $a$. Thus there are no sporadic solutions to Equation \eqref{eq:main} for $a>100$. Since the case  $a \leq 100$ has already been treated in Proposition \ref{prop:smallSols}, this completes the proof of  Theorem \ref{thm:main}.

Finally, let us remark that the computations described above were done in Sage \cite{Sage} and took about 5 minutes on a usual pc.

\section{Further problems}

In Theorem \ref{thm:main} we assumed for technical reasons $|n|\leq \max\{|a|^{1/3},1\}$. The question is, whether this assumption is necessary. 

For small $a$ there actually do exist further solutions if we omit this assumption. E.g. for $a=-1$ there is a total of 15 sporadic, pairwise non-equivalent, solutions with the largest $n$ being $22$.

However, a search in the range $- 1 \leq a \leq 400$ and $1\leq n \leq 400$ revealed no sporadic solutions for $a\geq 3$.
Therefore, we suspect that indeed a much stronger result holds.

\begin{con}\label{con:unit-eq}
The Diophantine equation 
\begin{equation*}
 u_1+u_2=n,
 \qquad 
 u_1,u_2\in \Z[\rho]^*, \;
 n\in \Z
\end{equation*}
has exactly 24 sporadic solutions that are pairwise non-equivalent.
Moreover, each sporadic solution $(u_1,u_2,n)$ satisfies $a\leq 2$ and $|n|\leq 22$.
\end{con}

Another natural question is whether our result can be extended to units of the maximal order  of $K_a$ instead of the order $\Z[\rho]$. In the range $-1\leq a \leq 400$ and $1\leq n \leq 400$ there are only 66 pairwise non-equivalent sporadic solutions and all of them satisfy $a\leq 66$ and $|n|\leq 22$.  
Therefore, we make the following (perhaps wild) conjecture.

\begin{con}\label{con:strong}
Unit Equation \eqref{eq:unit} has exactly 66 sporadic solutions that are pairwise non-equivalent. 
\end{con}

\section*{Acknowledgement}

We want to thank Attila B\'erczes for his help with the computations verifying Conjectures~\ref{con:unit-eq} and \ref{con:strong} in the range $- 1 \leq a \leq 400$ and $1\leq n \leq 400$.

\bibliographystyle{plain}
\bibliography{lit_unitEq}

\end{document}